\documentclass{amsart}
\usepackage{amssymb}
\usepackage{amsfonts}
\usepackage{mathrsfs}  
\usepackage{amsmath}
 \newtheorem{thm}{Theorem}[section]
 \newtheorem{cor}[thm]{Corollary}
 \newtheorem{lem}[thm]{Lemma}
 \newtheorem{prop}[thm]{Proposition}

 \newtheorem{rem}[thm]{Remark} 
 
 \numberwithin{equation}{section}
\renewcommand{\phi}{\varphi}
\renewcommand{\epsilon}{\varepsilon}
\newcommand{\U}{\mathbb{U}}

\newcommand{\bigoh}{\mbox{O}}
\newcommand{\bdu}{\partial \U}

\newcommand{\C}{\mathbb{C}}
\newcommand{\clspan}[1]{{\overline{\rm span}\,}\seq{#1}}
\newcommand{\const}{{\rm const.\,}}
\newcommand{\cphi}{C_\phi}
\newcommand{\cpsi}{C_\psi}
\newcommand{\cphidf}{\cphi |_{D_f}}

\newcommand{\eith}{e^{i\theta}}
\newcommand{\eps}{\varepsilon}

\newcommand{\goesto}{\rightarrow}

\newcommand{\onehalf}{\frac{1}{2}}
\newcommand{\htwo}{H^2}

\newcommand{\hol}[1]{\rm Hol\,(#1)}

\newcommand{\infsum}{\sum_{n=0}^\infty}

\newcommand{\inv}{^{-1}}
\newcommand{\invtwopi}{\frac{1}{2\pi}}
\newcommand{\intpi}{\int_{-\pi}^\pi}
\newcommand{\intbdu}{\int_{\bdu}}

\newcommand{\linspan}[1]{{\rm span}\,\{#1\}}

\newcommand{\norm}[1]{\|#1\|}

\newcommand{\Prn}{P_{r_n}}

\newcommand{\R}{\mathbb{R}}
\newcommand{\re}{{\rm Re\,}}
\newcommand{\rhp}{\rm \Pi^+}

\newcommand{\seq}[1]{\{#1\}}

\newcommand{\talpha}{\tilde{\alpha}}
\newcommand{\tbeta}{\tilde{\beta}}
\newcommand{\tphi}{\tilde{\phi}}

\newcommand{\Z}{\mathbb{Z}}

\begin{document}

\title[Hyperbolic composition operators]
 {Eigenfunctions for hyperbolic\\ composition operators---redux}
\author[J. H. Shapiro]{Joel H. Shapiro}

\address{%
Department of Mathematics and Statistics, 
Portland State University, 
Portland, OR 97207 
USA}

\email{shapiroj@pdx.edu}



\begin{abstract}
The Invariant Subspace Problem (``ISP'') for Hilbert space operators is known to be  equivalent to a question that, on its surface, seems surprisingly concrete: {\em For composition operators induced on the Hardy space $\htwo$ by hyperbolic automorphisms of the unit disc, is every nontrivial minimal invariant subspace one dimensional (i.e., spanned by an eigenvector)?}  In the hope of reviving interest in the contribution this remarkable result might offer to the studies of both composition operators and the ISP, I revisit some known results, weaken their hypotheses and simplify their proofs. {\em Sample results:} If $\phi$ is a hyperbolic disc automorphism with fixed points at $\alpha$ and $\beta$ (both necessarily on the unit circle), and $\cphi$ the composition operator it induces on $\htwo$, then for every $f\in\sqrt{(z-\alpha)(z-\beta)}\,\htwo$, the doubly $\cphi$-cyclic subspace generated by $f$ contains many independent eigenvectors; more precisely, the point spectrum  of $\cphi$'s restriction to that subspace  intersects the unit circle in a set of positive measure.
Moreover, this restriction of $\cphi$ is hypercyclic (some forward  orbit is dense). Under the stronger restriction $f\in\sqrt{(z-\alpha)(z-\beta)}\,H^p$  for some $p>2$, the point spectrum of the restricted operator contains an open annulus centered at the origin. 
\end{abstract}

\maketitle

\section{Introduction}\label{introduction}
More than twenty years ago  Nordgren, Rosenthal, and Wintrobe \cite{NRW} made a surprising connection between composition operators on the Hardy space $\htwo$ and the Invariant Subspace Problem---henceforth, the ``ISP''. The ISP asks if every operator on a separable Hilbert space has a nontrivial invariant subspace (following tradition: ``operator'' means ``bounded linear operator,'' ``subspace'' means ``closed linear manifold,'' and for a subspace, ``nontrivial'' means ``neither the whole space nor the zero-subspace'').  Nordgren, Rosenthal, and Wintrobe proved the following \cite[Corollary 6.3, page 343]{NRW}:  
\begin{quote}
 {\em Suppose $\phi$ is a hyperbolic automorphism of the open unit disc $\U$. Let $\cphi$ denote the composition operator induced by $\phi$ on the Hardy space $\htwo$. Then the ISP has a positive solution if and only if every nontrivial minimal $\cphi$-invariant subspace of $\htwo$ has dimension one.}
 \end{quote}
It is easy to see that, for each nontrivial minimal invariant subspace $V$ of a Hilbert space operator $T$, every non-zero vector $x\in V$ is cyclic, i.e., $\linspan{T^nx: n = 0, 1, 2, \ldots}$  is dense in $V$. If, in addition,  $T$ is invertible, then so is its restriction  to $V$ (otherwise the range of this restriction would be a nontrivial invariant subspace strictly contained in $V$, contradicting minimality). Thus for $T$ invertible, $V$ a nontrivial minimal invariant subspace of $T$, and $0\neq x \in V$,
$$
    V = \clspan{T^n x: n = 0, 1, 2, \ldots} = \clspan{T^n x: n\in\Z},
$$
where now ``$\overline{\rm span}$ '' means ``closure of the linear span.''

The result of Nordgren, Rosenthal, and Wintrobe therefore suggests that for $\phi$ a hyperbolic disc automorphism we might profitably study how the properties of a function $f$ in $\htwo\backslash\{0\}$ influence the operator-theoretic properties of $\cphidf$, the restriction of $\cphi$ to the ``doubly cyclic'' subspace
subspace
\begin{equation} \label{df_defn}
  D_f := \clspan{\cphi^n f: n\in\Z} = \clspan{f\circ\phi_n: n \in \Z},
\end{equation}
with particular emphasis on the question of when the point spectrum of the restricted operator is nonempty. (Here, for $n$ is a positive integer, $\phi_n$ denotes the $n$-th compositional iterate of $\phi$, while $\phi_{-n}$ is the  $n$-th iterate of $\phi\inv$;   $\phi_0$ is the identity map.)

Along these lines, Valentin Matache \cite[1993]{Mat} obtained a number of interesting results on minimal invariant subspaces for hyperbolic-automorphically induced composition operators. He observed, for example, that if a minimal invariant subspace for such an operator were to have dimension larger than $1$, then, at either of the fixed points of $\phi$, none of the non-zero elements of that subspace could be both continuous and non-vanishing (since $\phi$ is a hyperbolic automorphism of the unit disc, its fixed points must necessarily lie on the unit circle; see \S\ref{hypautos} below). Matache also obtained interesting results on the possibility of minimality for invariant subspaces generated by inner functions.
 
Several years later Vitaly Chkliar \cite[1996]{Chk} proved this result for hyperbolic-automorphic composition operators $\cphi$:
\begin{quote}
 {\em  If $f\in\htwo\backslash\{0\}$ is bounded in a neighborhood of one fixed point of $\phi$, and at the other fixed point vanishes to some order $\eps>0$, then the point spectrum of $\cphidf$ contains an open annulus centered at the origin.} 
\end{quote}
 Later Matache \cite{Mat2} obtained similar conclusions under less restrictive hypotheses. 
 
In the work below, after providing some background (in \S\ref{background}), I  revisit in \S \ref{main_results} and \S\ref{complements} the work of Chkliar and Matache,  providing simpler proofs of stronger results. Here is a sample: {\em for $\phi$ a hyperbolic automorphism of $\U$ with fixed points $\alpha$ and $\beta$ (necessarily on $\bdu$): }
  \begin{itemize}
\item[(a)] {\em If $f\in\sqrt{(z-\alpha)(z-\beta)}\,\htwo\backslash\{0\}$, then $\sigma_p(\cphidf)$ intersects the unit circle in a set of positive measure.} 
\item[(b)] {\em  If $f\in\sqrt{(z-\alpha)(z-\beta)}\,H^p\backslash\{0\}$  for some $p>2$, then $\sigma_p(\cphidf)$ contains an open annulus centered at the origin.}
\end{itemize}
 Note that the function $\sqrt{(z-\alpha)(z-\beta)}$ is an outer function, so the set of functions $f$ being singled out in both parts (a) and (b) is dense in $\htwo$.

 Finally, observe that, in the hypotheses of both (a) and (b), the exponent ``$\frac{1}{2}$'' is best possible in the sense that for any smaller exponent the function $f\equiv 1$, for which $D_f$ is the one dimensional subspace of constant functions, would satisfy both hypotheses. This comment applies throughout the sequel.
 
\subsection*{Acknowledgement} I wish to thank Professor Paul Bourdon  of Washington and Lee University for many suggestions that corrected and improved a preliminary version of this paper.

\section{Background material} \label{background}

\subsection{Disc automorphisms} \label{hypautos}
An {\em automorphism\/} of a domain in the complex plane is a univalent holomorphic mapping of that domain onto itself. Every automorphism of the open unit disc $\U$ is a linear fractional map \cite[Theorem 12.6, page 255]{Rud}. 
 
 Linear fractional maps can be regarded as  homeomorphisms of the Riemann Sphere; as such, each one that is not the identity map has one or two fixed points. The maps with just one fixed point are the {\em parabolic\/} ones; each such map is conjugate, via an appropriate linear fractional map, to one that fixes only the point at infinity, i.e., to a translation. A linear fractional map that fixes two distinct points is conjugate, again via a linear fractional map, to one that fixes both the origin and the point at infinity, i.e., to a dilation $w\goesto \mu w$ of the complex plane, where $\mu\neq 1$ is a complex number called the {\em multiplier\/} of the original map (actually $1/\mu$ can just as well occur as the multiplier---depending on which fixed point of the original map is taken to infinity by the conjugating transformation).  The original map is called {\em elliptic\/} if $|\mu|=1$, hyperbolic  if $\mu$ is positive, and {\em loxodromic\/} in all other cases (see, for example, \cite[Chapter 0]{S} for more details).

Suppose $\phi$ is a hyperbolic automorphism of $\U$. Then the same is true of its inverse. The fixed points of $\phi$ must necessarily lie on $\bdu$, the unit circle. Indeed,  if the attractive fixed point of $\phi$ lies outside the closed unit disc, then the compositional iterates of $\phi$ pull $\U$ toward that fixed point, and hence outside of $\U$,  which contradicts the fact that $\phi(\U)=\U$. If, on the other hand, the attractive fixed point lies in $\U$, then its reflection in the unit circle is the repulsive fixed point, which is the attractive one for $\phi^{-1}$. Thus $\phi^{-1}$ can't map $\U$ into itself, another contradiction. Conclusion: both fixed points lie on $\bdu$.

Let's call a hyperbolic automorphism $\phi$ of $\U$ {\em canonical\/} if it fixes the points $\pm 1$, with $+1$ being the attractive fixed point. We'll find it convenient to move between the open unit disc $\U$ and the open right half-plane $\rhp$ by means of the {\em Cayley transform\/} $\kappa:\Pi^+ \goesto \U$ and its inverse $\kappa\inv:\U\goesto \Pi^+$, where
$$
   \kappa(w) = \frac{w-1}{w+1} 
                 \quad {\rm and} \quad 
       \kappa\inv(z) = \frac{1+z}{1-z}                      
                   \qquad (z\in\U, w\in\Pi^+).        
$$
In particular, if $\phi$ is a canonical hyperbolic automorphism of $\U$, then $\Phi:=\kappa\inv\circ\phi\circ\kappa$ is an automorphism of $\rhp$ that fixes $0$ and $\infty$, with $\infty$ being the attractive fixed point. Thus $\Phi(w) = \mu w$ for some $\mu >1$, and $\phi = \kappa\circ\Phi\circ\kappa\inv$, which yields, after a little calculation,
\begin{equation} \label{canonical_auto}
  \phi(z) = \frac{r+z}{1+rz} \qquad {\rm where} 
                     \qquad \phi(0)= r = \frac{\mu-1}{\mu+1} \in (0,1).
\end{equation}
If $\phi$ is a hyperbolic automorphism of $\U$ that is not canonical, then it can be conjugated, via an appropriate automorphism of $\U$, to one that is. This is perhaps best seen by transferring attention to the right half-plane $\rhp$, and observing that if $\alpha <\beta$ are two real numbers, then the linear fractional map $\Psi$ of $\rhp$ defined by
$$
    \Psi(w) ~=~ i \, \frac{w-i\beta}{w-i\alpha}
$$
preserves the imaginary axis, and takes the point 1 into $\rhp$. Thus it is an automorphism of $\rhp$ that takes the boundary points $i\beta$ to zero and $i\alpha$ to infinity. Consequently if $\Phi$ is any hyperbolic automorphism of $\rhp$ with fixed points $i\alpha$ (attractive) and  $i\beta$ (repulsive), then $\Psi\circ\Phi\circ\Psi^{-1}$ is also hyperbolic automorphism with attractive fixed point $\infty$ and repulsive fixed point $0$. If, instead, $\alpha > \beta$ then $-\Psi$ does the job.
 
Since any hyperbolic automorphism $\phi$ of $\U$ is conjugate, via an automorphism, to a canonical one, $\cphi$ is similar, via the composition operator induced by the conjugating map, to a composition operator induced by a canonical hyperbolic automorphism. For this reason the work that follows will focus on the canonical case. 
 	
\subsection{Spectra of hyperbolic-automorphic composition operators}%
\label{spectra}

Suppose $\phi$ is a hyperbolic automorphism of $\U$ with multiplier $\mu>1$. Then it is easy to find lots of eigenfunctions/eigenvalues for $\cphi$ on $\htwo$. We may without loss of generality assume that $\phi$ is canonical, and then move, via the Cayley map, to the right half-plane where $\phi$ morphs into the dilation $\Phi(w) = \mu w$. Let's start by viewing the composition operator $C_\Phi$ as  just a linear map on $\hol{\rhp}$, the space of all holomorphic functions on $\rhp$. For any complex number $a$ define $E_a(w) = w^a$, where $w^a=\exp(a\log w)$, and ``log'' denotes the principal branch of the logarithm. Then $E_a\in\hol{\rhp})$ and $C_\Phi(E_a) = \mu^a E_a$, i.e., $E_a$ is an eigenvector of $C_\Phi$ (acting on $\hol{\rhp}$) and the corresponding eigenvalue is $\mu^a$ (again taking the principal value of the ``$a$-th power'').  Upon returning via the Cayley map to the unit disc, we see that, when viewed as a linear transformation of $\hol{\U}$, the operator $\cphi$ has, for each $a\in\C$, the eigenvector/eigenvalue combination $(f_a, \mu^a)$, where the function
\begin{equation} \label{eigenfunction_defn}
     f_a(z) = \left(\frac{1+z}{1-z}\right)^a \quad\qquad (z\in\U)
\end{equation}
belongs to $\htwo$ if and only if $|\re(a)| < 1/2$. Thus the corresponding $\htwo$-eigenvalues $\mu^a$ cover the entire open annulus 
\begin{equation} \label{spectral_annulus}
    A:= \{\lambda\in\C: \frac{1}{\sqrt{\mu}}<|\lambda|<\sqrt{\mu}\}.
\end{equation}
In particular $\sigma(\cphi)$, the $\htwo$-spectrum of $\cphi$, contains this annulus, and since the map $a\goesto \mu^a$ takes the strip $|\re(a)| < 1/2$ infinitely-to-one onto $A$, each point of $A$ is an eigenvalue of $\cphi$ having infinite multiplicity.

As for the rest of the spectrum, an elementary norm calculation shows that $\sigma(\cphi)$ is just the closure of $A_\mu$. To see this, note first that  the change-of-variable formula from calculus shows that for each $f\in\htwo$ and each automorphism $\phi$ of $\U$ (not necessarily hyperbolic):
\begin{equation}\label{ch_var}
    \norm{\cphi f}^2 = \intbdu |f|^2 P_a \, dm
\end{equation}
where $m$ is normalized arc-length measure on the unit circle $\bdu$, and $P_a$ is the Poisson kernel for $a=\phi(0)$; more generally, for any $a\in\U$: 
\begin{equation} \label{poisson_kernel_formula}
  P_a(\zeta) ~=~ \frac{1-|a|^2}{|\zeta - a|^2}  \qquad (\zeta\in\bdu) \, 
\end{equation}
(see also Nordgren's neat argument \cite[Lemma 1, page 442]{Nor}, which shows via Fourier analysis that \eqref{ch_var} holds for any inner function).

Now suppose $\phi$ is the canonical hyperbolic automorphism of $\U$ with multiplier $\mu>1$.   Then $\phi$ is given by \eqref{canonical_auto}, so by \eqref{poisson_kernel_formula}
$$
     P_r(\zeta) ~=~ \frac{1-r^2}{|\zeta-r|^2} ~\le~ \frac{1+r}{1-r} ~=~ \mu
$$
which, along with \eqref{ch_var} shows that 
\begin{equation}  \label{norm_ineq}
    \norm{\cphi} ~\le~ \sqrt{\mu} \, .
\end{equation}
Since also
$$
    P_r(\zeta) ~\ge~ \frac{1-r}{1+r} ~=~ \mu^{-1}
$$
we have, for each $f\in\htwo$
$$
    \norm{\cphi f} ~\ge~ \frac{1}{\sqrt{\mu}} \, \norm{f},
$$
which shows that \eqref{norm_ineq}  holds with $\cphi$ replaced by  $\cphi^{-1}$.  Thus the spectra of both $\cphi$ and its inverse lie in the closed disc of radius $\sqrt{\mu}$ centered at the origin, so by the spectral mapping theorem, 
    $\sigma(\cphi)$ is contained in the closure of the annulus \eqref{spectral_annulus}. Since we have already seen that this closed annulus contains the spectrum of $\cphi$ we've established the following result, first proved by Nordgren \cite[Theorem 6, page 448]{Nor} using precisely the argument given above:
  
  \begin{thm}  \label{spectrum_thm} 
   If $\phi$ is a hyperbolic automorphism of $\U$ with multiplier $\mu ~(>1)$, then $\sigma(\cphi)$ is the closed annulus
   $
  \{\lambda\in\C:1/\sqrt{\mu}\le |\lambda|\le \sqrt{\mu}\}.
   $
   The interior of this annulus consists entirely of eigenvalues of $\cphi$, each having infinite multiplicity.
  \end{thm} 
In fact the interior of $\sigma(\cphi)$ is precisely the point spectrum of $\cphi$; see \cite{Mat3} for the details.

\subsection{Poisson kernel estimates} \label{poisson_estimate}
Formula \eqref{poisson_kernel_formula}, giving the  Poisson kernel for the point $a=\rho e^{i\theta_0} \in\U$, can be rewritten
$$
    P_a(\eith) ~=~ \frac{1-\rho^2}{1-2\rho\cos(\theta-\theta_0) + \rho^2}
       \qquad (0\le\rho<1, \theta\in\R) \, .
$$
We will need the following well-known estimate, which provides a convenient replacement (cf. for example \cite[page 313]{A}).
\begin{lem} \label{poisson_est_lem}
For $0\le\rho<1$ and $|\theta|\le\pi$:
  \begin{equation} \label{poisson_int_est}
      P_\rho(\eith) ~\le~  4 \, \frac{(1-\rho)}{(1-\rho)^2+(\theta/\pi)^2}
   \end{equation} 
\end{lem}
\begin{proof}
 $$
     P_\rho(\eith) 
          := \frac{1-\rho^2}{1-2\rho\cos\theta+\rho^2}
          = \frac{1-\rho^2}{(1-\rho)^2+\rho(2\sin\frac{\theta}{2})^2} 
          \le \frac{2(1-\rho)}{(1-\rho)^2+4\rho(\theta/\pi)^2}    
 $$           
so, at least when $\rho\ge \frac{1}{4}$, inequality \eqref{poisson_int_est} holds with constant ``2'' in place of ``4''.  For the other values of $\rho$ one can get inequality \eqref{poisson_int_est} by checking that, over the interval $[0,\pi]$,  the minimum of the right-hand side exceeds the maximum of the left-hand side. 
\end{proof}
\begin{rem} \label{constant_remark}{\em
The only property of the constant ``4'' on the right-hand side of \eqref{poisson_int_est}  that matters for our purposes is its independence of $\rho$ and $\theta$.  
}
\end{rem}
For the sequel  (especially Theorem \ref{main_thm_2} below) we will require the following upper estimate of certain infinite sums of Poisson kernels. 

\begin{lem} \label{sum_estimate}
For $\phi$ the canonical hyperbolic automorphism of $\U$ with multiplier $\mu$:
\begin{equation} \label{poisson_sum_estimate}
  \infsum P_{\phi_n(0)}(\eith) 
               ~\le~ \frac{16\mu}{\mu-1} \, \frac{\pi}{|\theta|}
                  \qquad (|\theta|\le \pi) \,.
\end{equation}
\end{lem}
\noindent  In the spirit of Remark \ref{constant_remark} above,  the precise form of the positive constant that multiplies $\pi/|\theta|$  on the right-hand side of  \eqref{poisson_sum_estimate} is unimportant (as long as it does not depend on $\theta$). 

\begin{proof}
 The automorphism $\phi$ is  given by equations \eqref{canonical_auto}. For each integer $n\ge 0$ the $n$-th iterate $\phi_n$ of $\phi$ is just the canonical hyperbolic automorphism with multiplier $\mu^n$, so upon substituting $\mu^n$ for $\mu$ in \eqref{canonical_auto} we obtain
\begin{equation} \label{nth_iterate}
  \phi_n(z) = \frac{r_n+z}{1+r_nz} \qquad{\rm where} \qquad 
                  \phi_n(0) = r_n=\frac{\mu^n-1}{\mu^n+1} \in (0,1).
\end{equation}
Thus  $1-r_n = 2/(\mu^n+1)$,  and so
\begin{equation} \label{mu_power_est}
    \mu^{-n} ~<~ 1-r_n ~<~ 2 \, \mu^{-n}  \quad (n=0, 1, 2, ...),
\end{equation}
(in particular, $r_n$ approaches the attractive fixed point $+1$ with exponential speed as $n\goesto\infty$;  this is true of the $\phi$-orbit of any point of the unit disc). 

Fix $\theta\in[-\pi,\pi]$. We know from \eqref{poisson_int_est} and \eqref{mu_power_est}  that for each integer $n\ge 0$,
$$
   \Prn(\eith) ~\le~ \frac{4(1-r_n)}{(1-r_n)^2 + (\theta/\pi)^2}
                   ~\le~ \frac{8\mu^{-n}}{\mu^{-2n}+(\theta/\pi)^2}~,
$$
 whereupon, for each non-negative integer $N$:
 \begin{eqnarray*}
    \frac{1}{8} \infsum \Prn(\eith)  
          &\le& \infsum \frac{\mu^{-n}}{\mu^{-2n} + (\theta/\pi)^2} \\
          & &\\
          &\le& \sum_{n=0}^{N-1}\frac{\mu^{-n}}{\mu^{-2n}}
                       ~+~ \left(\frac{\pi}{\theta}\right)^2 
                                  \sum_{n=N}^\infty\mu^{-n} \\
                                  & &\\
           &=& \sum_{n=0}^{N-1} \mu^n
                               ~+~  \left(\frac{\pi}{\theta}\right)^2  
                              \sum_{n=N}^\infty \mu^{-n}   \\
             & &\\               
            &=& \frac{\mu^N - 1}{\mu - 1}  
                            ~+~ 
                          \left(\frac{\pi}{\theta}\right)^2 \mu^{-N} (1-\mu^{-1})^{-1}                                                        
 \end{eqnarray*}
where the geometric sum in the next-to-last line converges because $\mu>1$.

We need a choice of  $N$ that gives a favorable value for the quantity in the last line of the display above. Let $\nu = \log_\mu(\pi/|\theta|)$, so that $\mu^\nu = \pi/|\theta|$. Since $|\theta|\le\pi$ we are assured that $\nu\ge 0$. Let $N$ be the least integer $\ge \nu$, i.e., the unique integer in the interval $[\nu,\nu+1)$. 
The above estimate yields for any integer $N\ge 0$, upon setting $C:=8\mu/(\mu-1)$ (which is $>0$ since $\mu>1$),
\begin{equation} \label{first_sum_est}
  \infsum \Prn(\eith)     
       \le C\left[\mu ^{N-1} +  \left(\frac{\pi}{\theta}\right)^2 \mu^{-N} \right]  
      \le C  \frac{\pi}{|\theta|}
                     \left[\frac{|\theta|}{\pi}\mu^\nu 
                                + \left(\frac{|\theta|}{\pi}\mu^\nu\right)\inv\right].
\end{equation}
  By our choice of $\nu$,  both summands in the square-bracketed term at the end of \eqref{first_sum_est}   have the value 1  and this implies \eqref{poisson_sum_estimate}.
\end{proof}

\section{Main results} \label{main_results}
Here I extend work of Chkliar \cite{Chk} and Matache \cite{Mat2} that provides, for a hyperbolic-automorphically induced composition operator $\cphi$,  sufficient conditions on $f\in\htwo$ for the doubly-cyclic subspace $D_f$, as defined by (\ref{df_defn}), to contain  a rich supply of linearly independent eigenfunctions.  I'll focus mostly on canonical hyperbolic automorphisms, leaving the general case for the next section. Thus, until further notice, $\phi$ will denote a canonical hyperbolic automorphism of $\U$ with multiplier $\mu>1$, attractive fixed point at $+1$ and  repulsive one at $-1$, i.e., $\phi$ will be given by equations \eqref{canonical_auto}. 

Following both Chkliar and Matache, I will use an $\htwo$-valued Laurent series to produce the desired eigenvectors. The idea is this: for $f\in\htwo$, and $\lambda$ a non-zero complex number, if the series
\begin{equation} \label{eigen_series}
  \sum_{n\in\Z} \lambda^{-n} (f\circ\phi_n)
\end{equation}
converges strongly enough (for example, in $\htwo$) then the sum $F_\lambda$, whenever it is not the zero-function, will be a $\lambda$-eigenfunction of $\cphi$  that lies in $D_f$. Clearly the convergence of the series \eqref{eigen_series} will depend crucially on the behavior of $\phi$ at its fixed points, as the next result indicates. For convenience let's agree to denote by $A(R_1,R_2)$ the open annulus, centered at the origin, of inner radius $R_1$ and outer radius $R_2$ (where, of course, $0<R_1<R_2<\infty$).

\begin{thm} \label{main_thm_1} {\rm (cf. \cite{Chk})}
Suppose $0<\eps,\delta\le 1/2$, and that
$$
f\in (z-1)^{\frac{1}{2}+\eps} (z+1)^{\frac{1}{2}+\delta}\htwo\backslash\{0\}.
$$ 
Then $\sigma_p(\cphidf)$ contains, except possibly for a discrete subset, $A(\mu^{-\eps},\mu^\delta)$.
\end{thm}

\begin{proof}
Our hypothesis on the behavior of $f$ at the point $+1$ (the attractive fixed point of $\phi$) is that $f=(z-1)^{\frac{1}{2}+\eps}g$ for some $g\in\htwo$, i.e., that
\begin{equation} \label{boundary_hypothesis}
  \infty ~>~ \intbdu |g|^2dm ~=~  \intbdu 
               \frac{|f(\zeta)|^2}{|\zeta-1|^{2\eps+1}}\,dm(\zeta) 
         ~\ge~ \invtwopi \intpi 
                      \frac{|f(\eith)|^2}{~~|\theta|^{2\eps+1}} \, d\theta \, .
\end{equation} 
Upon setting $a=\phi_n(0) := r_n$ in  \eqref{ch_var} we obtain
  \begin{equation} \label{ch_var_htwo}
      \norm{f\circ\phi_n}^2 ~=~ \int |f|^2 P_{r_n}\,dm \, ,
                 \qquad (n\in\Z)
  \end{equation}
which combines with estimates \eqref{poisson_int_est} and \eqref{mu_power_est} to show that if $n$ is a non-negative integer (thus insuring that $r_n > 0$):
\begin{eqnarray*}
    \norm{f\circ\phi_n}^2 
          &\le &  2\pi \intpi
                       |f(\eith)|^2 \frac{1-r_n}{(1-r_n)^2+\theta^2} \,d \theta \\
           & &\\            
          &\le& 4\pi\intpi  |f(\eith)|^2
                         \frac{\mu^{-n}}{\mu^{-2n}+\theta^2} \, d\theta      \\
           & &\\              
          &=& 4\pi \mu^{-2n\eps} \intpi 
                               \frac{|f(\eith)|^2}{|\theta|^{1+2\eps}} \, 
                \left\{\frac{(\mu^n|\theta|)^{1+2\eps}}{1+(\mu^n|\theta|)^2}
                                             \right\} \, d\theta\\
            & &\\                                 
           &\le& 4\pi  \mu^{-2n\eps}  
                    \sup_{x\in\R}\left\{ \frac{|x|^{1+2\eps}}{1+x^2}\right\}  
                         \intpi \frac{|f(\eith)|^2}{|\theta|^{1+2\eps}} \, d\theta  ~.     
\end{eqnarray*}
By \eqref{boundary_hypothesis} the integral in the last line is finite, and since $0<\eps\le 1/2$, the supremum in that line is also finite. Thus 
$$
     \norm{f\circ\phi_n} ~=~ \bigoh(\mu^{-n\eps}) 
                \quad {\rm as}\quad n\goesto\infty,
$$
which guarantees that the subseries of \eqref{eigen_series} with positively indexed terms converges in $\htwo$ for all $\lambda\in\C$ with $|\lambda| > \mu^{-\eps}$. 

As for the negatively indexed subseries of \eqref{eigen_series},  note from \eqref{canonical_auto} that $\phi\inv(z) = -\phi(-z)$, so  $\phi_{-n}(z) = -\phi_n(-z)$ for each integer $n$. Let $g(z) = f(-z)$, so our hypothesis on $f$ implies that $g\in (z-1)^{\onehalf+\delta}\,\htwo\backslash\{0\}$. Let $\psi_n(z) = \phi_n(-z)$ (the subscript on $\psi$ does not now indicate iteration). Then for each positive integer $n$ we have $\psi_n(0)=\phi_n(0) = r_n$, hence:
$$
    \norm{f\circ\phi_{-n}}^2 
           ~=~ \norm{g\circ\psi_n}^2 ~=~ \intbdu |g|^2 \Prn \, dm
$$
so by the result just obtained, with $g$ in place of $f$ and $\eps$ replaced by $\delta$,
$$          
           \norm{f\circ\phi_{-n}} ~=~ \bigoh(\mu^{-n\delta})  
               \quad {\rm as} \quad n\goesto\infty.
$$  
Thus the negatively indexed subseries of \eqref{eigen_series}   converges in $\htwo$ for all complex numbers $\lambda$        with $|\lambda|<\mu^\delta$. 

Conclusion: For each $\lambda$ in the open annulus $A(\mu^{-\eps}, \mu^\delta)$ the $\htwo$-valued Laurent series \eqref{eigen_series} converges in the norm topology of $\htwo$ to a function $F_\lambda \in \htwo$. Now $F_\lambda$, for such a $\lambda$, will be a $\cphi$-eigenfunction unless it is the zero-function, and---just as for scalar Laurent series---this inconvenience can occur for at most a discrete subset of points $\lambda$ in the annulus of convergence (the relevant uniqueness theorem for $\htwo$-valued holomorphic functions follows easily from the scalar case upon applying bounded linear functionals).
\end{proof}

\begin{rem}{\em 
Chkliar \cite{Chk} has a similar result, where there are uniform conditions on the function $f$ at the fixed points of $\phi$ (see also Remark \ref{Chkliar_remark} below); as he suggests,  it would be of interest to know whether or not the ``possible discrete subset'' that clutters the conclusions of results like Theorem \ref{main_thm_1} can actually be nonempty. 
}
\end{rem}
\begin{rem}{\em
The limiting case $\delta=0$ of Theorem \ref{main_thm_1} still holds (see Theorem \ref{main_thm_3} below); it is a slight improvement on  Chkliar's result (see also the discussion following Theorem \ref{main_thm_3}).
}
\end{rem}
\begin{rem} {\em
Note that the restriction $\eps, \delta \le 1/2$ in the hypothesis of Theorem \ref{main_thm_1} cannot be weakened since, as mentioned at the end of \S\ref{spectra}, the point spectrum of  $\cphi$ is the open annulus $A(\mu^{-\onehalf}, \mu^\onehalf)$. 
}
\end{rem}

Here is a companion to Theorem \ref{main_thm_1}, which shows that even in the  limiting case $\delta=\eps=0$ (in some sense the ``weakest'' hypothesis on $f$) the operator $\cphidf$ still has a significant supply of eigenvalues. 

 \begin{thm} \label{main_thm_2}
 If $f\in\sqrt{(z+1)(z-1)} \, \htwo$ then $\sigma_p(\cphi|_{D_f})$ intersects $\bdu$ in a set of positive measure.
 \end{thm}
\begin{proof}
We will work in the Hilbert space $L^2(\htwo,dm)$ consisting of $\htwo$-valued ($m$-equivalence classes of) measurable functions $F$ on $\bdu$ with
$$
     |||F|||^2 ~:=~ \int_{\bdu} \|F(\omega)\|^2 dm(\omega) <\infty.
$$
I will show in a moment that the hypothesis on $f$ implies 
\begin{equation} \label{series_conv}
  \sum_{n\in\Z} \norm{f\circ\phi_n}^2 < \infty \,.
\end{equation}
Granting this, it is easy to check that the $\htwo$-valued Fourier series
\begin{equation} \label{vector_fs}
  \sum_{n\in\Z} (f\circ\phi_n) \,  \omega^{-n} \qquad (\omega\in\bdu)
\end{equation}
converges unconditionally in $L^2(\htwo,dm)$, so at least formally, we expect that for a.e. $\omega\in\bdu$ we'll have $\cphi(F(\omega)) = \omega F(\omega)$. This is true, but a little care is needed to prove it. The ``unconditional convergence'' mentioned above means this: If, for each finite subset $E$ of $\Z$, 
$$
     S_E(\omega) := \sum_{n \in E} (f\circ\phi_n) \omega^{-n} 
                             \qquad (\omega\in\bdu) \, ,
$$
then the net 
$
     (S_E: E~{\rm a~finite~subset~of}~ \Z)
$
  converges in $L^2(\htwo, dm)$ to $F$. In particular, if for each non-negative integer $n$ we define  $F_n = S_{[-n,n]}$,  then $F_n\goesto F$ in $L^2(\htwo,dm)$, hence some subsequence  $(F_{n_k}(\omega))_{k=1}^\infty$ converges in $\htwo$ to $F(\omega)$ for a.e. $\omega\in\bdu$. Now for any $n$ and any $\omega\in\bdu$:
  $$
     \cphi F_n(\omega) = \omega F_n(\omega) 
                                   - \omega^{n+1}f\circ\phi_{-n}
                                   +\omega^{-n} f\circ\phi_{n+1}                                       
  $$
 which implies, since \eqref{series_conv} guarantees that  $\norm{f\circ\phi_n}\goesto 0$ as $n\goesto\infty$, that
 $$
        \cphi F_n(\omega) - \omega F_n(\omega)  
          \goesto 0~~{\rm in}~~\htwo \qquad (n\goesto\infty).
  $$
This, along with the a.e. convergence of the subsequence $(F_{n_k})$ to $F$, shows that $\cphi F(\omega) = \omega F(\omega)$ for a.e. $\omega\in\bdu$. Now the $\htwo$-valued Fourier coefficients $f\circ\phi_n$ are not all zero (in fact, none of them  are zero) so at least for a subset of points $\omega\in\bdu$ having positive measure we have $F(\omega)\neq 0$. The corresponding $\htwo$-functions $F(\omega)$ are therefore eigenfunctions of $\cphi$ that belong to $D_f$, thus $\sigma_p(\cphidf)\cap\bdu$ has positive measure.

It remains to prove \eqref{series_conv}. As usual, we treat the positively and negatively indexed terms separately. Since $f\in\sqrt{z-1}\, \htwo$ we have
$$
    \frac{1}{2\pi}\int_{-\pi}^\pi \frac{|f(\eith)|^2}{|\theta|} d\theta
         ~\le~ \int_{\bdu} \frac{|f(\zeta)|^2}{|\zeta-1|} dm(\zeta)
         ~<~\infty
$$
so successive application of  \eqref{ch_var} and \eqref{poisson_sum_estimate} yields
$$
    \infsum\norm{f\circ\phi_n}^2
            =  \intbdu |f|^2 \left(\infsum \Prn \right)dm 
            ~\le~ \const  \intpi \frac{|f(\eith)|^2}{|\theta|}\, d\theta
            ~<~ \infty \, .
$$

For the negatively indexed terms in \eqref{series_conv}, note that our hypothesis on $f$ guarantees that
\begin{equation} \label{second_hyp}
    \frac{1}{2\pi}\int_{-\pi}^\pi \frac{|f(e^{i(\theta-\pi)})|^2}{|\theta|} d\theta
         ~\le~ \int_{\bdu} \frac{|f(\zeta)|^2}{|\zeta+1|} dm(\zeta)
         ~<~ \infty \, .
\end{equation}
Recall from the proof of Theorem \ref{main_thm_1} that 
$\phi_{-n}(z) = -\phi_n(-z)\,$ for $z\in\U$ and $n>0$, and so
$$
    \norm{f\circ\phi_{-n}}^2 = \int |f|^2P_{-r_n} dm
          = \frac{1}{2\pi} \int_{-\pi}^\pi |f(\eith)|^2P_{r_n}(\theta-\pi)\,d\theta.
$$
Thus
\begin{eqnarray*}
   \sum_{n=1}^\infty\norm{f\circ\phi_{-n}}^2 
         &=&
                \frac{1}{2\pi} \int_{-\pi}^\pi |f(\eith)|^2 
                       \left(\sum_{n=1}^\infty P_{r_n}(\theta-\pi)\right) d\theta \\
          &=&
                 \frac{1}{2\pi} \int_{-\pi}^\pi |f(e^{i(\theta-\pi)})|^2 
                       \left(\sum_{n=1}^\infty P_{r_n}(\theta)\right) d\theta  \\
           & &\\            
           & \le& \const\, \int_{-\pi}^\pi 
                        \frac{|f(e^{i(\theta-\pi)})|^2}{|\theta|} d\theta \\
                        & &\\
            &<& \infty                     
\end{eqnarray*}
where the last two lines follow, respectively,  from inequalities \eqref{poisson_sum_estimate} and \eqref{second_hyp}.  This completes the  proof of \eqref{series_conv}, and with it, the proof of the Theorem.
\end{proof} 

It would be of interest to know just how large a set $\sigma_p(\cphidf)$ has to be in Theorem \ref{main_thm_2}. Might it always be the whole unit circle? Might it be even larger?  What I do know is that if the hypothesis of the Theorem is strengthened by replacing the hypothesis ``$f\in \sqrt{(z+1)(z-1)}\,\htwo$\,'' with the stronger ``$f\in \sqrt{(z+1)(z-1)}\,H^p$ for some $p>2$\,'', then the conclusion improves dramatically, as shown below by the result below, whose proof reprises the latter part of the proof of Theorem \ref{main_thm_1}.

\begin{thm} \label{main_thm_3} {\rm (cf. \cite[Theorem 5.5]{Mat2})}
If $f\in \sqrt{(z+1)(z-1)}\, H^p\backslash\{0\}$ for some $p>2$, then $\sigma_p(\cphidf)$ contains, except possibly for a discrete subset, the open annulus $A(\mu^{-\eps},\mu^\eps)$ where $\eps = \onehalf-\frac{1}{p}$.
\end{thm}
\begin{proof}
 I will show that the   hypothesis implies that
 $f\in[(z-1)(z+1)]^{\onehalf+\delta}\htwo$ for each positive $\delta<\eps$. This will guarantee, by the proof of Theorem \ref{main_thm_1}, that the series \eqref{eigen_series} converges in the open annulus $A(\mu^{-\delta},\mu^\delta)$ for each such $\delta$, and hence it converges in $A(\mu^{-\eps},\mu^\eps)$, which will, just as in the proof of Theorem \ref{main_thm_1} finish the matter. The argument below, suggested by Paul Bourdon, greatly simplifies my original one.
Our hypotheses of $f$ imply that for some $g\in H^p$,
$$
f=[(z-1)(z+1)]^{\onehalf + \delta} h \quad 
     {\rm  where} \quad h = [(z-1)(z+1)]^{-\left(\onehalf + \delta\right)} g.
$$
To show: $h\in\htwo$. The hypothesis on $\delta$ can be rewritten: $2p\delta/(p-2)<1$, so the function $[(z-1)(z+1)]^{-\delta}$ belongs to $H^{\frac{2p}{p-2}}$, hence an application of H\"older's inequality shows that $h$ is in $\htwo$ with norm bounded by the product of the $H^p$-norm of $g$ and the $H^{\frac{2p}{p-2}}$-norm of $[(z-1)(z+1)]^{-\delta}$.
\end{proof}

In both \cite{Chk} and \cite[Theorem 5.3]{Mat2} there are results where the hypotheses on $f$ involve uniform boundedness for $f$ at one or both of the fixed points of $\phi$. In \cite[Theorem 5.4]{Mat2} Matache shows that these uniform conditions can be replaced by boundedness of a certain family of Poisson integrals, and from this he derives the following result. 
\begin{quote}
{\em
\cite[Theorem 5.5]{Mat2} If $f\in (z-1)^{\frac{2}{p}}\,H^p$ for some $p>2$, and $f$ is bounded in a neighborhood of $-1$, then $\sigma_p(\cphidf)$ contains an open annulus centered at the origin.
}
\end{quote}
I'll close this section by presenting some results of this type, where uniform boundedness at one of the fixed points is replaced by boundedness of the Hardy-Littlewood maximal function. This is the function, defined for $g$  non-negative and integrable on $\bdu$, and $\zeta\in\bdu$,  by:
$$
     M[g](\zeta) :=  \sup\left\{\frac{1}{m(I)}\int_I g \, dm: I 
       ~{\rm an~arc~of}~ \bdu~{\rm centered~at}~\zeta\right\}.
$$
The {\em radial maximal function\/} $R[g]$ of $g$ at $\zeta\in\bdu$ is the supremum of the values of the Poisson integral of $g$ on the radius  $[0,\zeta)$. It is easy to check that $M[g]$ is dominated pointwise on $\bdu$ by a constant multiple of $R[g]$. What is perhaps surprising, but still elementary, is the fact that there is a similar inequality in the other direction: the radial maximal function of the non-negative integrable function $g$ is dominated pointwise on $\bdu$ by a constant multiple of its Hardy-Littlewood maximal function (see \cite[Theorem 11.20, page 242]{Rud}). This and \eqref{ch_var} yield
\begin{lem} \label{max_fcn_lemma}
   For  $f\in\htwo$,
   $$
       M[|f|^2](-1) < \infty 
                ~~\implies~~ 
                     \sup\{\norm{f\circ\phi_n}: n<0\} < \infty.                                      
   $$
\end{lem}
To see that the hypotheses of Lemma \ref{max_fcn_lemma} can be satisfied by functions in $\htwo$ that are unbounded as $z\goesto -1$, one need only observe that 
$$
    M[|f|^2](-1) ~\le~   \const\,\int \frac{|f(\zeta)|^2}{|1+\zeta|}\,dm(\zeta) \, ,
$$
hence, along with \eqref{ch_var}, the Lemma implies:
\begin{cor} \label{zplusone_cor}
If $f\in\sqrt{z+1} \, \htwo$ ~then~ $\sup\{\norm{f\circ\phi_n}: n<0\} < \infty$. 
\end{cor}

Thus if $f\in\sqrt{z+1}\,\htwo$, or more generally if $M[|f|^2](-1)<\infty$, the negatively indexed subseries of \eqref{eigen_series} will converge in $\htwo$ for all $\lambda\in\U$. We have seen in the proof of Theorem \ref{main_thm_1} that if $f\in (z-1)^{\onehalf + \eps}\,\htwo$ for some $\eps\in(0,1/2]$ then the positively indexed subseries of \eqref{eigen_series} converges for $|\lambda| > \mu^{-\eps}$. Putting it all together we obtain the promised ``$\delta=0$'' case of Theorem \ref{main_thm_1}:
\begin{thm} \label{main_thm_4} 
Suppose $f\in (z+1)^{\onehalf}(z-1)^{\onehalf+\eps}\,\htwo\backslash\{0\}$ for some $0<\eps<1/2$. Then $\sigma_p(\cphidf)$ contains, with the possible exception of a discrete subset, the open annulus $A(\mu^{-\eps}, 1)$. 
\end{thm}
\begin{rem} \label{Chkliar_remark}{\em
By the discussion preceding this theorem, the hypothesis on $f$ could be replaced by the weaker: ``$f\in (z-1)^{\onehalf+\eps}\,\htwo\backslash\{0\}$ and $M[|f|^2](-1)<\infty$, '' (cf. \cite{Chk}). If,  in either version, the hypotheses on the attractive and repulsive fixed points are reversed, then the conclusion will assert that $\sigma_p(\cphidf)$ contains, except for perhaps a discrete subset, the annulus $A(1, \mu^\eps)$ (see \S\ref{non-canonical}, especially the discussion preceding  Corollary \ref{inverse_cor}). 
}
\end{rem}
\begin{rem}{\em
Note how the previously mentioned Theorem 5.5 of \cite{Mat2} follows from the work above. Indeed, if $f\in (z-1)^{2/p}\,H^p$ for some $p>2$ then by H\"older's inequality  $f\in (z-1)^{\onehalf+\eps}\,\htwo$, for each $\eps<1/p$. Thus, as in  the proof of Theorem \ref{main_thm_1},  the positively indexed subseries of \eqref{eigen_series} converges for $|\lambda|>\mu^{-1/p}$, and  by Lemma \ref{max_fcn_lemma} the boundedness of $f$ in a neighborhood of $-1$ insures that the negatively indexed subseries of \eqref{eigen_series} converges in the open unit disc. Thus as in the proof of Theorem \ref{main_thm_1},  $\sigma_p(\cphidf)$ contains, with the possible exception of a discrete subset,  the open annulus $A(\mu^{-1/p},1)$.
}
\end{rem}

\section{Complements and comments} \label{complements}

In this section I collect some further results and say a few more words about the theorem of Nordgren, Rosenthal, and Wintrobe.

\subsection{Non-canonical hyperbolic automorphisms} \label{non-canonical}
The results of \S\ref{main_results}, which refer only to canonical hyperbolic automorphisms $\phi$, can be easily ``denormalized''. Here is a sample:

\begin{thm} \label{non_canonical_thm}
Suppose $\phi$ is a hyperbolic automorphism of $\U$ with attractive fixed point $\alpha$, repulsive one $\beta$, and multiplier $\mu>1$. Then 
\begin{itemize}
   \item[(a)] {\rm (cf. Theorem \ref{main_thm_1})} Suppose, for $0<\eps,\delta<1/2$ we have 
$$
f\in (z-\alpha)^{\frac{1}{2}+\eps} (z-\beta)^{\frac{1}{2}+\delta}\htwo\backslash\{0\}.
$$ 
Then $\sigma_p(\cphidf)$ contains, except possibly for a discrete subset, the open annulus $A(\mu^{-\eps},\mu^\delta)$.
    \item[(b)] {\rm (cf. Theorem \ref{main_thm_2})}  If $f\in\sqrt{(z-\alpha)(z-\beta)} \, \htwo$ then $\sigma_p(\cphi|_{D_f})$ intersects $\bdu$ in a set of positive measure.
    \item[(c)]{\rm  (cf. Theorem \ref{main_thm_3})} If $f\in \sqrt{(z-\alpha)(z-\beta)}\, H^p\backslash\{0\}$ for some $p>2$, then $\sigma_p(\cphidf)$ contains, except possibly for a discrete subset, the open annulus $A(\mu^{-\eps},\mu^\eps)$ where $\eps = \onehalf-\frac{1}{p}$.        
\end{itemize}
\end{thm}
\begin{proof}
I'll just outline the idea, which contains no surprises. Suppose $\alpha$ and $\beta$ (both on $\bdu$)  are the fixed points of $\phi$, and---for the moment---that $\talpha$ and $\tbeta$ are any two distinct points of $\bdu$. Then, as we noted toward the end of \S\ref{hypautos}, there is an automorphism  $\psi$ of $\U$ that takes $\talpha$ to $\alpha$ and $\tbeta$ to $\beta$. Thus $\tphi := \psi\inv\circ\phi\psi$ is a hyperbolic automorphism of $\U$ that is easily seen to have attractive fixed point $\talpha$ and repulsive one $\tbeta$. Furthermore:
\begin{itemize}
   \item $C_{\tphi} = \cpsi \cphi \cpsi\inv$, 
            so $C_{\tphi}$  is similar to $\cphi$.
   \item For $f\in\htwo$: $\cpsi D_f = D_{f\circ\psi}$.
   \item $F\in\htwo$ is a $\lambda$-eigenvector for $\cphi$ if and only if 
            $\cpsi =F\circ\psi$ is one for $C_{\tphi}$. 
   \item For $f\in\htwo$, $M[|f|^2](\beta)<\infty 
             \iff M[|f\circ\psi|^2](\beta)<\infty. $        
    \item For any $\gamma>0$, $f\in (z-\alpha)^\gamma\, \htwo 
             \iff \cpsi f \in (z-\talpha)^\gamma \, \htwo$.       
\end{itemize}
Only the last of these needs any comment. If $f\in (z-\alpha)^\gamma\, \htwo$ then
  \begin{eqnarray*} 
      \cpsi f  &\in& (\psi(z) - \alpha)^\gamma \cpsi(\htwo) \\
                 & &\\
                 &=&  \left(\frac{\psi(z)- \psi(\talpha)}{z-\talpha}\right)^\gamma
                         (z-\talpha)^\gamma \htwo \\
                  & &\\       
                  &=& (z-\talpha)^\gamma \, \htwo       
  \end{eqnarray*} 
where the last line follows from the fact that the quotient in the previous one is, in a neighborhood of the closed unit disc, analytic and non-vanishing (because $\psi$ is univalent there), hence both bounded and bounded away from zero on the closed unit disc. Thus $\cpsi((z-\alpha)^\gamma\,\htwo) \subset (z-\talpha)\, \htwo$, and the opposite inclusion follows from this by replacing $\psi$ by $\psi\inv$ and applying $\cpsi$ to both sides of the result.
  
  Theorem \ref{non_canonical_thm} now follows, upon setting $(\talpha,\tbeta) = (+1,-1)$, from Theorems \ref{main_thm_1}, \ref{main_thm_2}, and  \ref{main_thm_3}.
  \end{proof}
What happens if we interchange attractive and repulsive fixed points of $\phi$ in the hypotheses of Theorem \ref{non_canonical_thm}(a)? Then the hypotheses apply to $\phi\inv$, hence so does the conclusion. Since $C_{\phi\inv} = \cphi\inv$, Theorem \ref{non_canonical_thm}(a) and the spectral mapping theorem yield, for example, the following complement to Theorem \ref{main_thm_4}:
\begin{cor} \label{inverse_cor}
  If $\phi$ is a hyperbolic automorphism of $\U$ with attractive fixed point $\alpha$, repulsive one $\beta$, and multiplier $\mu>1$.  Suppose $f \in (z-\alpha)^\frac{1}{2}(z-\beta)^{\frac{1}{2}+\eps}$ for some $\eps\in(0,\frac{1}{2})$, then $\sigma_p(\cphi |_{D_f})$ contains, except possibly for a discrete subset, the open annulus $A(1,\mu^{\eps})$
\end{cor}
The reader can easily supply similar ``reversed'' versions of the other results on the point spectrum of $\cphi |_{D_f}$. 

\subsection{The Nordgren-Rosenthal-Wintrobe Theorem} \label{RNW_thm}
Recall that this result equates a positive solution to the Invariant Subspace Problem for Hilbert space with a positive answer to the question: ``For $\phi$ a hyperbolic automorphism of $\U$, does does every nontrivial minimal $\cphi$-invariant subspace of $\htwo$ contain an eigenfunction?''  The theorem comes about in this way: About forty years ago Caradus \cite{Car} proved the following elementary, but still remarkable, result:
\begin{quote}{\em
 If an operator $T$ maps a separable, infinite dimensional Hilbert space onto itself and has infinite dimensional null space, then {\em every\/} operator on a separable Hilbert space is similar to a scalar multiple of the restriction of $T$ to one of its invariant subspaces. 
 }
 \end{quote}
 Consequently the invariant subspace lattice of $T$ contains that of every operator on a separable Hilbert space. 

Now all composition operators (except the ones induced by constant functions) are one-to-one, so none of these obeys the Caradus theorem's hypotheses. However Nordgren, Rosenthal, and Wintrobe were able to show that if $\phi$ is a hyperbolic automorphism, then for every eigenvalue $\lambda$ of $\cphi$ the operator $\cphi-\lambda I$, which has infinite dimensional kernel (recall Theorem \ref{spectrum_thm}), maps $\htwo$ onto itself. Their restatement of the Invariant Subspace Problem follows from this via the Caradus theorem and the fact that $\cphi$ and $\cphi-\lambda I$ have the same invariant subspaces.

\subsection{Cyclicity}  \label{cyclicity}
Minimal invariant subspaces for invertible operators are both cyclic and doubly invariant---this was the original motivation for studying the subspaces $D_f$. Thus it makes sense, for a given doubly invariant subspace, and especially for a doubly cyclic one $D_f$,  to ask whether or not it is cyclic. Here is a result in that direction in which the cyclicity is the strongest possible: {\em hypercyclicity}---some orbit (with no help from the linear span) is dense.  I state it for canonical hyperbolic automorphisms; the generalization to non-canonical ones follows from the discussion of \S\ref{non-canonical} and the similarity invariance of the property of hypercyclicity. 

\begin{prop}  \label{cyclic_prop}
Suppose $\phi$ is a canonical hyperbolic automorphism of $\U$ and $f\in\sqrt{(z+1)(z-1)} \, \htwo$. Then $\cphidf$ is hypercyclic.
\end{prop}
\begin{proof}
A sufficient condition for an invertible operator on a Banach space $X$ to be hypercyclic is that for some dense subset of the space, the positive powers of both the operator and its inverse tend to zero pointwise in the norm of $X$ (see \cite[Chapter 7, page 109]{S}, for example; much weaker conditions suffice). In our case the dense subspace is just the linear span of $S:=\{f\circ\phi_n: n\in\Z\}$. As we saw in the proof of Theorem \ref{main_thm_2}, our  hypothesis on $f$ insures that  $\sum_{n\in\Z}\norm{f\circ\phi_n}^2 <\infty$ so both $(\cphi^n)_0^\infty$ and $(\cphi^{-n})_0^\infty$ converge pointwise to zero on $S$, and therefore pointwise on its linear span. 
\end{proof}

\begin{rem} \label{hc_remark}{\em
One can obtain the conclusion of Proposition  \ref{cyclic_prop} under different hypotheses. For example if $f$ is continuous with value zero at both of the fixed points of $\phi$, then the same is true of the restriction of $|f|^2$ to $\bdu$. Thus the Poisson integral of $|f|^2$ has radial limit zero at each fixed point of $\phi$ (see \cite[Theorem 11.3, page 244]{Rud}, for example), so by \eqref{ch_var_htwo}, just as in the proof of Proposition \ref{cyclic_prop}, $\cphidf$ satisfies the sufficient condition for hypercyclicity. In fact, all that is really needed for this argument is that the measure 
$$
    E\goesto\int_E |f|^2 \, dm \qquad (E~{\rm measurable}\subset\bdu)
$$     
have symmetric derivative zero at both fixed points of $\phi$ (see the reference above to \cite{Rud}).
}
\end{rem}


\end{document}